\documentclass{amsart}

\usepackage{amsmath}
\usepackage{amssymb}
\usepackage{mathrsfs}
\usepackage{url}
\usepackage{enumitem}
\usepackage{stmaryrd}
\usepackage[all]{xy}
\usepackage{verbatim}

\newtheorem{Thm}{Theorem}[section]

\newtheorem{Cor}[Thm]{Corollary}
\newtheorem{Def}[Thm]{Definition}
\newtheorem{DefThm}[Thm]{Definition-Theorem}
\newtheorem{Eg}[Thm]{Example}
\newtheorem{Lem}[Thm]{Lemma}
\newtheorem{Prop}[Thm]{Proposition}
\newtheorem{Qn}[Thm]{Question}
\newtheorem{Rmk}[Thm]{Remark}

\numberwithin{equation}{Thm}

\newcommand{\FF}{\mathbb{F}}

\newcommand{\NN}{\mathbb{N}}

\newcommand{\OO}{\mathscr{O}}
\newcommand{\plog}{\text{Log}}
\newcommand{\Qp}{\mathbb{Q}_p}
\newcommand{\QQ}{\mathbb{Q}}

\newcommand{\Tr}{\text{Tr}}
\newcommand{\WW}{\mathbb{W}}

\newcommand{\Zp}{\mathbb{Z}_p}
\newcommand{\ZZ}{\mathbb{Z}}
\newcommand{\PP}{\mathbb{P}}

\author{Christopher Davis}
\address{University of California, Irvine, Dept of
Mathematics, Irvine, CA 92697}
\email{davis@math.uci.edu}
\date{\today}

\author{Daqing Wan}
\address{University of California, Irvine, Dept of
Mathematics, Irvine, CA 92697}
\email{dwan@math.uci.edu}

\title{L-functions of $p$-adic characters}

\begin{document}

\begin{abstract}
We define a $p$-adic character to be a continuous homomorphism from $1 + t\FF_q[[t]]$ to $\ZZ_p^*$.  We use the ring of big Witt vectors over $\FF_q$ to exhibit a bijection between $p$-adic characters and sequences $(c_i)_{(i,p) = 1}$ of elements in $\ZZ_q$, indexed by natural numbers relatively prime to $p$, and for which $\lim_{i \rightarrow \infty} c_i = 0$.  To such a $p$-adic character we associate an $L$-function, and we prove that this $L$-function is $p$-adic meromorphic if the corresponding sequence $(c_i)$ is overconvergent.  If more generally the sequence is $c\log$-convergent, we show that the associated $L$-function is meromorphic in the open disk of radius $q^c$.  Finally, we exhibit examples of $c\log$-convergent sequences with associated $L$-functions which are not meromorphic in the disk of radius $q^{c+\epsilon}$ for any $\epsilon > 0$.
\end{abstract}

\maketitle

\section{Introduction}

 Let $\FF_q$ be a finite field of $q$ elements with characteristic $p$. Let $K=\FF_q(t)$ be the rational function field, 
 which is the function field of the projective line $\PP^1$ over $\FF_q$. Let $G_K$ denote the absolute Galois group of $K$, 
 namely, the Galois group of a fixed separable closure of $K$. Given a continuous $p$-adic representation 
 $$\rho: G_K \longrightarrow {\rm GL}_n(\ZZ_p),$$
  unramified on $U=\PP^1 -S$ with $S$ being a finite set of closed points of $\PP^1$, the L-function of the representation $\rho$ is defined by 
  $$L(\rho, s) := L(\rho/U, s) = \prod_{x\in |U|} \frac{1}{\det(I - \rho({\rm Frob}_x) s^{\deg(x)})} \in 1 +s \ZZ_p[[s]],$$
where $|U|$ denotes the set of closed points of $U$ and ${\rm Frob}_x$ denotes the geometric Frobenius conjugacy class at $x$. 
It is clear that the power series $L(\rho, s)$ is convergent (or analytic) in the open unit disk $|s|_p <1$. 

 A basic object of study in number theory is the L-function $L(\rho, s)$. The first  question about $L(\rho, s)$ is its possible analytic or 
 meromorphic continuation. This question has been studied extensively in the literature, however, it remains quite mysterious, even in the abelian case when rank $n=1$. 
We now briefly review the limited known results.  If $\rho$ is of finite order, then $L(\rho, s)$ is a rational function in $s$ (Brauer) and it satisfies 
a Riemann hypothesis (Weil). If,  in addition,  $\rho$ is irreducible and non-trivial, then $L(\rho, s)$ is a polynomial in $s$ (Artin's conjecture 
for function fields), which follows from Grothendieck's trace formula. 

If $\rho$ is a $p$-adic representation of infinite order, the situation is much more  complicated, even in the abelian case $n=1$. 
First, it is easy to construct examples such that $L(\rho, s)$ is not rational in $s$. For an arbitrary $p$-adic representation $\rho$, 
the L-function $L(\rho, s)$ is known to be meromorphic on the closed unit disc 
$|s|_p \leq 1$ and its unit root part (a rational function in $s$) is given by the Frobenius action on the $p$-adic \`etale cohomology of $\rho$. This was 
conjectured by Katz \cite{Katz73}, and proved by Crew \cite{Crew87} in the rank $n=1$ case, and in general proved by Emerton-Kisin \cite{EK01}. 
A stronger conjecture of Katz \cite{Katz73} stated that $L(\rho, s)$ is meromorphic in $|s|_p < \infty$.  This turned out to be false in general, even in the case $n=1$; see \cite{Wan96}. It suggests that 
the L-function $L(\rho, s)$ is much more complicated than what was thought before.

Motivated by his pioneering work on $p$-adic variation of zeta functions, Dwork \cite{Dwork71,Dwork73} conjectured that if $\rho$ is geometric 
(arising from the relative $p$-adic \'etale cohomology of a family of varieties), then the L-function $L(\rho, s)$ is $p$-adic meromorphic in $|s|_p < \infty$. This was proved 
by the second author \cite{Wan99,Wan001,Wan002}. It suggests that the class of geometric $p$-adic representations behaves reasonably well from the L-function point of view. 
We note that even in the geometric rank one case, although the L-function $L(\rho, s)$ is $p$-adic meromorphic in $|s|_p <\infty$, 
it is not expected to be rational in $s$, nor should one expect that it is $p$-adic entire (namely, the Artin entireness conjecture 
fails for non-trivial rank one geometric $p$-adic representations of $G_K$). One such example follows from Coleman's work \cite{Coleman97} in the elliptic modular case.

The aim of this paper is to re-examine this L-function from a new point of 
view via Witt vectors in the hope that it will provide new insight into this mysterious meromorphic continuation problem. 
We shall focus on the abelian case $n=1$. Then the representation $\rho$ factors through the maximal abelian quotient $G_K^{\rm ab}$: 
 $$\rho: G_K^{\rm ab}  \longrightarrow {\rm GL}_1(\ZZ_p) =\ZZ_p^*.$$
That is, $\rho$ is a $p$-adic character. By class field theory, $G_K^{\rm ab}$ is isomorphic to the profinite completion of the id\`ele class group of 
$K$. Precisely, in our case of the rational function field $K=\FF_q(t)$, we have 
$$G_K^{\rm ab} \equiv  \widehat{ \left\langle t \right\rangle}    \times \left(1 +\frac{1}{t}\FF_q\left[\left[\frac{1}{t}\right]\right]\right) \times \prod_{x \in |\mathbb{A}^1|} \FF_q[t]_x^*,$$
where $\FF_q[t]_x$ denotes the completion of $\FF_q[t]$ at the prime $x$ and $\widehat{ \left\langle t \right\rangle}$ denotes the profinite 
completion of the infinite cyclic multiplicative group generated by $t$. Since the character $\rho$ is unramified on $U = \PP^1 -S$, the 
restriction of $\rho$ to the $x$-factor $\FF_q[t]_x^*$ is trivial for all $x \in U$. To further simplify the situation, we shall assume that 
$S$ is the one point set consisting of the origin corresponding to the prime $t$ in $\FF_q[t]$. In this case, $\FF_q[t]_t = \FF_q[[t]]$. 
Twisting by a harmless finite character, we may further assume that $\rho$ factors through the character 
$$\chi: \FF_q[[t]]^*/\FF_q^* = 1 +t\FF_q[[t]] \longrightarrow \ZZ_p^*.$$
If $f(t) \in 1 +t\FF_q[t]$ is an irreducible polynomial, then one checks that 
$$\rho({\rm Frob}_{f(t)}) = \chi(f(t)).$$
Thus, the L-function $L(\rho, s)$ reduces to the following L-function of the $p$-adic character $\chi$
$$L(\chi, s) =\prod_{f} \frac{1}{1 -\chi(f)s^{\deg (f)}},$$
where $f$ now runs over all monic irreducible polynomials of $\FF_q[t]$ different from $t$. 
Expanding the product, the L-function of $\chi$ is also the following series 
$$L(\chi, s) =\sum_{g} \chi(g) s^{\deg (g)} \in 1+s \ZZ_p[[s]],$$
where $g$ runs over all monic polynomials in $\FF_q[t]$ different from $t$ (alternatively, one defines $\chi(t)=0$). 
A related function, which is of great interest to us,  is the 
characteristic series of $\chi$ defined by 
$$C(\chi, s) =\prod_{k=0}^{\infty} L(\chi, q^ks)  \in 1 +s \ZZ_p[[s]],$$
Equivalently, 
$$L(\chi, s) =\frac{C(\chi, s)}{C(\chi, qs)}.$$
Thus,  the L-function and the characteristic series determine each other.

In summary, the question that we study in this paper becomes the following completely self-contained 

\begin{Qn} Given a continuous $p$-adic character 
$$\chi: 1 +t\FF_q[[t]] \longrightarrow \ZZ_p^*,$$
when is its L-function $L(\chi, s)$ as defined above $p$-adic meromorphic in $s$? 
\end{Qn}

To give an idea of what we prove, we state our result in this introduction only in the simpler special case that $q=p$. 

\begin{Thm} There is a one-to-one correspondence between continuous $p$-adic characters 
$\chi: 1 +t\FF_p[[t]] \longrightarrow \ZZ_p^*$ and sequences $\pi =(\pi_i)_{(i,p)=1}$, 
where $\pi_i \in p\ZZ_p$ and $\lim_{i \rightarrow \infty} \pi_i=0$. Denote the associated character by $\chi_{\pi}$. For an irreducible polynomial $f(t) \in 1 +t \FF_p[t]$ 
with degree $d$, let $\bar{\lambda}$ denote a reciprocal root of $f(t)$ and let $\lambda$ denote the Teichm\"uller lifting of $\bar{\lambda}$. 
Then the character $\chi_{\pi}$ is given by 
$$\chi_{\pi}(f(t)) = \prod_{(i,p)=1} (1+\pi_i)^{Tr(\lambda^i)},$$
where $Tr$ denotes the trace map from $\ZZ_{p^d}$ to $\ZZ_p$. Assume that the sequence $\pi$ satisfies the $\infty\log$-condition 
$$\liminf_{i\rightarrow \infty}  \frac{ v_p(\pi_i)}{\log_p i} = \infty.$$
Then the characteristic power series 
$$C(\chi, s) = \prod_{k=0}^{\infty} L(\chi, p^ks)$$ 
is entire in $|s|_p < \infty$ and thus the 
$L$-function 
$$L(\chi_{\pi}, s)=\frac{C(\chi, s)}{C(\chi, ps) }
$$
is $p$-adic meromorphic in $|s|_p < \infty$. 
\end{Thm}

Furthermore, we show that the theorem is optimal in the sense that the $\infty\log$-condition cannot be weakened to a $c\log$-condition 
for any finite $c$. To prove the above theorem, we link the character $\chi_{\pi}$ via the binomial 
power series to a power series in $\lambda$ with a good convergence condition and then apply the results in \cite{Wan96}. 
Thus, our proof ultimately depends on Dwork's trace formula. It would be very interesting to find a self-contained proof of the 
above theorem without using Dwork's trace formula, as this would pioneer an entirely new (and likely motivic) approach. 

In the present paper, we only treat L-functions of the simplest nontrivial $p$-adic characters, that is, $p$-adic 
characters with values in $\ZZ_p^*$ ramified only at the origin. There are several interesting ways to extend 
the present work. One can consider $p$-adic characters ramified at several closed points (not just the origin $t$). 
One can consider $p$-adic characters with values in the unit group of other $p$-adic rings such as the 
two dimensional local ring $\ZZ_p[[T]]$. (This will be very useful in studying the variation of the L-function when 
the character $\chi$ moves in a $p$-adic analytic family; here $T$ is the analytic parameter.  
Two related examples which have been studied in depth are the eigencurves \cite{CM98} and $T$-adic L-functions \cite{LW09}.) One can replace 
the projective line $\PP^1$ by a higher genus curve or even a higher dimensional variety. One can also consider higher rank $p$-adic representations, instead of considering only $p$-adic characters. 
We hope to return to some 
of these further questions in later papers.

\subsection*{Notation and conventions}  Let $q = p^a$ denote a power of $p$.  Beginning in Section \ref{In terms of Witt section} we require $p > 2$.  For a ring $R$ (always assumed commutative and with unity), we denote by $W(R)$ the $p$-typical Witt vectors with coefficients in $R$, and we denote by $\WW(R)$ the big Witt vectors with coefficients in $R$.  For an explanation of these Witt vectors, see Section \ref{Witt intro}.  We write $\ZZ_p$ for the ring of $p$-adic integers, $\QQ_p$ for the field of $p$-adic numbers, $\QQ_q$ for the unramified degree $a$ extension of $\QQ_p$, $\ZZ_q$ for the ring of integers in $\QQ_q$, $\widehat{\QQ_p^{nr}}$ for the $p$-adic completion of the maximal unramified extension of $\QQ_p$, and $\widehat{\ZZ_p^{nr}}$ for the ring of integers in $\widehat{\QQ_p^{nr}}$.  When we have a fixed unramified extension $\QQ_q/\QQ_p$, we write $\sigma$ for the Frobenius map, the unique automorphism which induces the Frobenius in characteristic $p$.  It is a generator of the cyclic group $\text{Gal}(\QQ_q/\QQ_p)$.   We let $f_{\lambda}(t)$ denote a general irreducible polynomial in $1 + t\FF_q[t]$.  We call its degree $d$, we write $\overline{\lambda}$ for one of its reciprocal roots, and we write $\lambda$ for the corresponding Teichm\"uller lift in $\ZZ_{q^d}$.  In other words, $\lambda$ will denote the unique root of unity in $\ZZ_{q^d}$ with reduction modulo $p$ equal to $\overline{\lambda}$.  We write $v_p(x)$ to denote the $p$-adic valuation of $x$.  If $R$ is a topological ring, we let $R\langle t \rangle$ denote convergent power series with coefficients in $R$.  Unfortunately we will need both the $p$-adic logarithm and the classical base-$p$ logarithm.  We denote the $p$-adic logarithm by $\plog$ and the classical logarithm by $\log_p$.  

\section{Preliminaries}

In this section we introduce the objects we will study ($p$-adic characters and their associated $L$-functions) and we introduce one of the key tools we will use to study them (big Witt vectors).

\subsection{$p$-adic characters} \label{padic intro}  We begin by defining $p$-adic characters and their associated $L$-functions.

\begin{Def}
By \emph{$p$-adic character}, we will mean a nontrivial continuous homomorphism
\[
\chi: (1 + t \FF_q[[t]])^* \rightarrow \ZZ_p^*.
\]
\end{Def}

\begin{Rmk} \label{Rmk continuity}
When we refer to continuity in the preceding definition, we are using the $t$-adic topology on $(1 + t \FF_q[[t]])^*$ and the $p$-adic topology on $\ZZ_p^*$.  We check that if $y \in \ZZ_p^*$ is in the image of $\chi$, then $y \equiv 1 \mod p$; we call such an element a \emph{1-unit}.  Assume $x := 1 + tf(t)$ is some element of $(1 + t \FF_q[[t]])^*$, and let $y := \chi(x)$.  The reduction $y \mod p \in \FF_p$ is a unit, so we have $y \equiv \zeta \mod p$, where $\zeta$ is some $(p - 1)$-st root of unity.  The sequence $x, x^p, x^{p^2}, \ldots$ clearly converges $t$-adically to $1$.  The sequence $y, y^p, y^{p^2}, \ldots$ converges $p$-adically to $\zeta$.  Hence, if the character $\chi$ is to be continuous, we must have $\zeta = \chi(1) = 1$.  This shows that the image of $\chi$ contains only $1$-units.
\end{Rmk}

To a $p$-adic character, we associate an $L$-function as follows.  

\begin{Def} \label{L-function Def}
The $L$-function associated to a $p$-adic character $\chi$ is the formal power series associated to either and both of the following:
\begin{align}
L(\chi, s) &= \prod_{{f(t) \text{ irred poly,} \atop f(t) \equiv 1 \bmod t}} \frac{1}{1 - \chi(f(t)) s^{\text{deg }f}} \\
\intertext{or}
L(\chi, s) &= \exp \left( \sum_{k = 1}^{\infty} \frac{S_k(\chi)}{k} s^k \right), \label{EltEqn}
\end{align}
where 
\[
S_k(\chi) = \sum_{\overline{\lambda} \in \FF_{q^k}} \chi(f_{\lambda})^{\frac{k}{\deg\left(\overline{\lambda}\right)}}, \]
and where $f_{\lambda}$ denotes the irreducible polynomial with constant term $1$ and having ${\overline{\lambda}}$ as a reciprocal root. 
\end{Def}
The most basic question to ask about these functions is the following.

\begin{Qn}
When is $L(\chi,s)$ a $p$-adic meromorphic function in the variable $s$?
\end{Qn}

To approach this question, we will use Witt vectors to give a new characterization of $p$-adic characters.  The next section reviews Witt vectors and states the results we will use.  Later we apply these results to give a simple description of all $p$-adic characters.

\subsection{Big and $p$-typical Witt vectors} \label{Witt intro}
We now introduce big and $p$-typical Witt vectors.  These are both functors from rings to rings.  Witt vectors are used in the proofs of our main theorems, but the statements of the theorems do not require Witt vectors.  The reader who is unfamiliar with Witt vectors should focus on the special cases concerning characteristic $p$, as that is what we will need below.  Our references for this section are \cite{Hes05} and \cite{Ser79}, but there are many other places to read about Witt vectors.  See for example the book of Hazewinkel \cite[\S 17]{Haz78} or the exercises of \cite[Chapter 9]{Bou83}.  We also recommend the notes of Lenstra \cite{Len02} and the notes of Rabinoff \cite{Rab07}.

The big Witt vector functor has an imposing definition, but when it is evaluated on a perfect field of characteristic $p$, as it will be in our case, it is quite accessible.  

\begin{Def}
Let $R$ denote a ring.  The ring of \emph{big Witt vectors with coefficients in $R$}, denoted $\WW(R)$, is, as a set, $R^{\NN} = \{(r_1, r_2, \ldots) \mid r_i \in R\}$.  To uniquely describe $\WW(R)$ as a ring, we demand the following two properties.  
\begin{enumerate}
\item The \emph{ghost map} $w: \WW(R) \rightarrow R^{\NN}$ defined by \[(r_1,r_2,\ldots,r_i,\ldots) \mapsto (r_1, r_1^2 + 2r_2, \ldots, \sum_{d \mid i} dr_d^{i/d}, \ldots)
\] is a ring homomorphism, where the ring operations on the target are component-wise.
\end{enumerate}  
This uniquely determines $\WW(R)$ as a ring in the case that $R$ is $\ZZ$-torsion free.  To determine the ring operations in general, we need also the following functoriality property.
\begin{enumerate}[resume]
\item For any ring homomorphism $f: R \rightarrow S$, the map $\WW(f): \WW(R) \rightarrow \WW(S)$ given by $\WW(f): (r_1,r_2,\ldots) \mapsto (f(r_1),f(r_2),\ldots)$ is a ring homomorphism.
\end{enumerate}
\end{Def}

\begin{Rmk}
It requires proof that such a functor $\WW(-)$ exists.  See for instance Proposition 1.2 in \cite{Hes05}.
\end{Rmk}

The more classical version of Witt vectors are the \emph{$p$-typical Witt vectors}.  Again, the general definition may be imposing, but when evaluated on a perfect field of characteristic $p$, there is a down-to-earth description.

\begin{Def}
Let $R$ denote a ring.  The ring of \emph{$p$-typical Witt vectors with coefficients in $R$}, denoted $W(R)$, is, as a set, $R^{\NN} = \{(r_1, r_p, r_{p^2}, \ldots) \mid r_{p^i} \in R\}$.  The ring operations on $W(R)$ are again defined using the ghost map.
\begin{enumerate}
\item The \emph{ghost map} $w: W(R) \rightarrow R^{\NN}$ defined by \[(r_1,r_p,\ldots,r_{p^i},\ldots) \mapsto (r_1, r_1^p + pr_p, \ldots, \sum_{d \mid p^i} dr_d^{p^i/d}, \ldots)
\] is a ring homomorphism, where the ring operations on the target are component-wise.
\end{enumerate} 
This determines the ring operations on $W(R)$ uniquely when $R$ is $p$-torsion free.  The definition in general follows by making a functoriality requirement as in the big case.
\end{Def}

\begin{Rmk}
Note that big Witt vectors are written using a boldface $\WW$, while $p$-typical Witt vectors are written using a roman $W$.  Also, note that from the definition of ring operations in terms of the ghost map, it is clear that $W(R)$ is a quotient of $\WW(R)$, but it is not a subring.
\end{Rmk}

When the ring $R$ is perfect of characteristic $p$, we have the following classical definition of $p$-typical Witt vectors, taken from Serre \cite[Theorem II.5.3.]{Ser79}.

\begin{Def}
If $k$ is a perfect field in characteristic $p$, then the ring $W(k)$ is the unique (up to canonical isomorphism) $p$-adically complete discrete valuation ring with maximal ideal $(p)$ and residue field $k$.
\end{Def}  

\begin{Eg}
Thus $W(\FF_p) = \Zp$ and $W(\FF_{p^a}) = \ZZ_{p^a}$ (which by definition is the ring of integers in the unramified extension of $\Qp$ of degree $a$).  Similarly, if $\overline{\FF_p}$ denotes the algebraic closure of $\FF_p$, then $W(\overline{\FF_p}) = \widehat{\ZZ_p^{nr}}$.
\end{Eg}

We are now ready to give the simpler definition of big Witt vectors $\WW(R)$ when $R$ is perfect of characteristic $p$, or, more generally, a $\ZZ_{(p)}$-algebra.  (The latter is the same as requiring that all integers relatively prime to $p$ are invertible in $R$.)

\begin{Prop} \label{Prop idempotent}
Let $R$ denote a $\ZZ_{(p)}$-algebra.  Then we have an isomorphism of rings 
\[
\WW(R) \cong \prod_{i \in \NN, (i,p) = 1} W(R).
\] 
\end{Prop}

\begin{proof}
For a proof, see \cite[Proposition 1.10]{Hes05}.  The idea is to prescribe ghost components of many mutually orthogonal idempotents, and then to use the fact that all primes $l \neq p$ are invertible in $R$ to find Witt vectors with the prescribed ghost components. 
\end{proof}

\begin{Rmk}
Let $U_j \subseteq \WW(R)$ denote the ideal $U_j = \left\{\underline{x} \in \WW(R) \mid x_i = 0 \text{ for } i < j  \right\}$.  These ideals generate a topology on $\WW(R)$ called the $V$-adic topology.  Under the isomorphism in Proposition \ref{Prop idempotent}, these ideals correspond to $U_j' \subseteq \prod_{(i,p) = 1} W(R)$, where 
\[U_j' = \left\{(\underline{x}_i) \in \prod_{(i,p) = 1} 
W(R) \mid \text{ the $p^k$ component of $\underline{x}_i$ equals zero for all  $ip^k < j$} \right\}.\]
\end{Rmk}

\begin{Eg}
It follows immediately from the proposition that as rings
\[
\WW(\FF_q) \cong \prod_{i \in \NN, (i,p) = 1} \ZZ_q.
\]
We are also interested in the topology on these rings.  Viewing an element $\alpha \in \ZZ_q$ as a Witt vector $\underline{x} \in W(\FF_q)$, we have that $p^i \mid \alpha$ if and only if the first $i$ coordinates of $\underline{x}$ are zero: $x_1,\ldots,x_{p^{i-1}} = 0$.  (This follows from the fact that multiplication by $p$ corresponds to shifting the Witt vector coordinates to the right and raising each component to the $p$-th power.  See \cite[Lemma 1.12]{Hes05} or \cite[Section II.6]{Ser79}.)  It is now easy to see that the $V$-adic topology described above is the same as the product topology on $\prod \ZZ_q$, where each component is equipped with the $p$-adic topology.
\end{Eg}

The following is yet another description of $\WW(R)$.  It is not as simple as the above description, but it hints at how we will use Witt vectors in our work on $p$-adic characters.

\begin{DefThm} \label{DefThm Lambda}
For any ring $R$, let $\Lambda(R) := 1 + tR[[t]]$ consist of power series with coefficients in $R$ with constant term $1$.  View $\Lambda(R)$ as a group under multiplication.  View $\WW(R)$ as a group under addition.  Then the map $E: \WW(R) \rightarrow \Lambda(R)$ defined by
\[
E: (r_1,r_2,\ldots) \mapsto \prod_{i = 1}^{\infty} (1 - r_i t^i) 
\]
is a group isomorphism.  If we view $\WW(R)$ as a topological group using the $V$-adic topology described above, and we view $\Lambda(R)$ as a topological group using the $t$-adic topology, then the isomorphism is a homeomorphism.
\end{DefThm}

\begin{Rmk}
There are four different reasonable normalizations for $E$.  These can be obtained by replacing $1 - r_i t^i$ above with $(1 \pm r_i t^i)^{\pm 1}$.  We have chosen the normalization which gives us easiest access to reciprocal roots.
\end{Rmk}

\begin{proof}
See \cite[Proposition 1.14]{Hes05} for everything except the continuity claims, and these are obvious.  The proof there is for a different normalization, but this doesn't matter.
\end{proof}

In either $p$-typical or big Witt vectors, we have a notion of Teichm\"uller lift.

\begin{Def}
If $r \in R$ is any element, we let $[r] \in W(R)$ or $\WW(R)$ denote the Witt vector with components $(r,0,0,\ldots)$.  This Witt vector is called the Teichm\"uller lift of $r$.  We have $[rs] = [r][s]$, but of course $[\cdot]: R \rightarrow W(R)$ is not a ring homomorphism (since traditionally $R$ is characteristic $p$ and $W(R)$ is characteristic zero).
\end{Def}

\begin{Rmk}
For a nonzero element $x \in \FF_{q}$, the Teichm\"uller lift $[x] \in \ZZ_q$ is the unique $(q-1)$-st root of unity which is a lift of $x$.  This follows from the equalities $[x]^{q-1} = [x^{q-1}] = [1] = 1$.
\end{Rmk}

\begin{Rmk}
In general it is not so easy to describe explicitly the map $\WW(R) \rightarrow \prod_{(i,p) = 1} W(R)$ from Proposition \ref{Prop idempotent}.  However, for the case of Teichm\"uller lifts, we can describe it explicitly.  In that case, we have $[\lambda] \mapsto ([\lambda^i])_{(i,p) = 1}$.  To prove this, work in terms of ghost components, using the explicit description from \cite[Proposition 1.10]{Hes05}.
\end{Rmk}

\begin{Eg} \label{Eg commutative diagram}
We have natural inclusions $1+t\FF_q[[t]] \subseteq 1+t \overline{\FF_q}[[t]]$ and $\WW(\FF_q) \subseteq \WW(\overline{\FF_q})$.  The latter corresponds to the inclusion
\[
\prod_{(i,p) = 1} \ZZ_q \subseteq \prod_{(i,p) = 1} \widehat{\ZZ_p^{\text{nr}}}.
\]
This yields a commutative diagram 
\[
\xymatrix{
1+t\FF_q[[t]] \ar[r]\ar[d] & 1 + t\overline{\FF_q}[[t]] \ar[d] & (1-\overline{\lambda}t) \ar[d] \\
\prod_{(i,p) = 1} \ZZ_q \ar[r] & \prod_{(i,p) = 1} \widehat{\ZZ_p^{\text{nr}}} & (\lambda,\lambda^2,\ldots),
}
\]
where we write $\lambda$ for the Teichm\"uller lift of $\overline{\lambda}$.

Consider now an irreducible degree $d$ polynomial $f(t) \in 1 + t\FF_q[t]$ with reciprocal root $\overline{\lambda}$.  This polynomial can be factored as $(1 - \overline{\lambda} t) (1 - \overline{\lambda}^q t) \cdots (1 - \overline{\lambda}^{q^{d-1}} t).$  The conjugates of the Teichm\"uller lift $\lambda$ over $\ZZ_q$ are $\lambda, \lambda^q, \ldots, \lambda^{q^{d-1}}$.  (Proof.  They are roots of unity and have the correct reductions modulo $p$.)  The above commutative diagram then shows that the polynomial $f(t)$ corresponds to the element $(\Tr(\lambda^i))_{(i,p) = 1} \in \prod_{(i,p) = 1} \ZZ_q$, where $\Tr$ denotes the trace from $\ZZ_{q^d}$ to $\ZZ_q$.
\end{Eg}

\section{Characters in terms of Witt vectors} \label{In terms of Witt section}

We begin this section with a result which gives us a new description of $p$-adic characters.  It follows directly from combining Definition-Theorem \ref{DefThm Lambda} with Proposition \ref{Prop idempotent}.  

\begin{Cor} \label{Witt corollary}
Giving a continuous homomorphism $\chi: (1+t\FF_q[[t]])^* \rightarrow \ZZ_p^*$ is the same as giving a continuous homomorphism 
\[
\prod_{i \in \NN, (i,p) = 1} \ZZ_q \rightarrow \ZZ_p^*.
\]
\end{Cor}

A first step towards understanding such continuous homomorphisms is to understand the component homomorphisms $\ZZ_q \rightarrow \ZZ_p^*$.  Before describing these, we describe the homomorphisms $\ZZ_q \rightarrow \ZZ_p$.

\begin{Lem} \label{Trace Lem}
For any continuous group homomorphism $\phi: \ZZ_q \rightarrow \ZZ_p$, there exists a unique element $c \in \ZZ_q$ such that $\phi: \alpha \mapsto \Tr(c\alpha)$, where $\Tr$ is the trace map from $\ZZ_q$ to $\ZZ_p$.  
\end{Lem}

\begin{proof}
We know from \cite[Section II.5]{Ser79} that as topological groups $\ZZ_q \cong \oplus_{i=1}^{a} \ZZ_p$. 
Fix a basis $e_1, \ldots, e_a$ of $\ZZ_q$ as a $\ZZ_p$-module.  Write $\phi(e_i) = r_i$.   We want to find $c \in \ZZ_q$ such that for any $b \in \ZZ_q$, we have $\phi(b) = \Tr(cb)$.  Write $b = b_1 e_1 + \cdots + b_a e_a$, where each $b_i \in \ZZ_p$.  Writing $c = c_1 e_1 + \cdots + c_a e_a$, our goal now becomes to find the elments $c_i \in \ZZ_p$.  We want
\begin{align*}
\phi(b) &= \Tr(cb) \\
\phi(b_1 e_1 + \cdots + b_a e_a) &= \Tr(cb) \\
b_1 r_1 + \cdots + b_a r_a &= \sum c_i b_j \Tr(e_i e_j).
\end{align*}
Considering the case $b_j = 1$ and $b_i = 0$ for $i \neq j$, we see that we must find $c_i$ so that 
\begin{equation} \label{matrix equation}
r_j = \sum_{i = 1}^a c_i \Tr(e_i e_j).
\end{equation}
In fact, if we find such $c_i$, then we are done.  (Simply compare the coefficients of $b_i$ above.)  We now show that $c_i$ satisfying (\ref{matrix equation}) exist and are uniquely determined.

Reducing everything modulo $p$, we know that $\overline{e_1},\cdots,\overline{e_a}$ is a basis for $\FF_{q}/\FF_p$.  This is a separable extension, and so 
the matrix with $(i,j)$-entry $\Tr(\overline{e_i} \overline{e_j})$ is invertible.  (See for example page 50 of \cite{Ser79}.)  Hence the determinant of the matrix with $(i,j)$-entry $\Tr(e_i e_j)$ is nonzero modulo $p$, hence it is invertible in $\ZZ_p$.  Hence we can find $c_i \in \ZZ_p$ which satisfy (\ref{matrix equation}) for all $j$, and these $c_i$ are uniquely determined. 
\end{proof}

We will study maps $\ZZ_q \mapsto \ZZ_p^*$ by factoring them as $\ZZ_q \rightarrow \ZZ_p \rightarrow \ZZ_p^*$.  The previous lemma concerned the first map in this composition, the next lemma concerns the second map.

\begin{Lem} \label{Lem (1+p)^r}
Assume $p > 2$.  Let $\pi \in p^n\ZZ_p$ for $n \geq 1$.  There exists a unique element $r \in p^{n-1}\ZZ_p$ such that $1 + \pi = (1 + p)^r$.
\end{Lem}

\begin{proof}
It suffices to show that for any $c \in \{0,1,\ldots,p-1\}$, we can find unique $r \in p^{n-1} \ZZ_p/p^n \ZZ_p$ such that $(1+p)^r \equiv 1 + cp^n \mod p^{n+1}$.  Using that $p \neq 2$, we see from the binomial expansion that only $r = cp^{n-1} \mod p^n \ZZ_p$ works.  
\end{proof}

We now realize our first goal of characterizing all continuous group homomorphisms from $\ZZ_q$ to $\ZZ_p^*$.  We will then be able to join these together to describe all continuous group homomorphisms from $\prod_{(i,p) = 1} \ZZ_q$ to $\ZZ_p^*$.  The following proposition will relate this to the case $q = p$.

\begin{Prop} \label{rho factor}
Assume $p > 2$.
Any continuous group homomorphism $\chi: \ZZ_q \rightarrow \ZZ_p^*$ can be factored as 
\[
\chi' \circ \Tr \circ c: \ZZ_q \stackrel{\cdot c}{\rightarrow} \ZZ_q \stackrel{\Tr}{\rightarrow} \ZZ_p \stackrel{\chi'}{\rightarrow} \ZZ_p^*,
\] where $c$ denotes multiplication by some element $c \in \ZZ_q$ and where $\chi': \ZZ_p \rightarrow \ZZ_p^*$.  Conversely, any such factorization yields a continuous group homomorphism.  

Moreover, we will see that $\chi'$ can be taken to be the map $\alpha \mapsto (1+p)^{\alpha}$.  With this restriction, then the corresponding element $c$ is unique.
\end{Prop}

\begin{proof}
That any such composition yields a continuous homomorphism is clear because all maps in the composition are continuous homomorphisms (for instance, trace is a sum of continuous homomorphisms).

Now consider any continuous homomorphism $\chi: \ZZ_q \rightarrow \ZZ_p^*$.  By Lemma \ref{Lem (1+p)^r}, there exists a unique map $\phi: \ZZ_q \rightarrow \ZZ_p$ such that $\chi(\alpha) = (1+p)^{\phi(\alpha)}$ for any $\alpha \in \ZZ_q$.  By our assumptions on $\chi$, it is easy to see that $\phi$ is a (continuous group) homomorphism.  Hence we are done by Lemma \ref{Trace Lem}.
\end{proof}

\begin{Prop} \label{41712Prop}
Let $q = p^a$, where $p$ is an \emph{odd} prime.  Giving a $p$-adic character 
\[
\chi: (1 + t\FF_q[[t]])^* \rightarrow \ZZ_p^* 
\]
is equivalent to giving a sequence of elements $(c_i)_{(i,p) = 1}$, where each $c_i \in \ZZ_q$, subject to the constraint that $\lim_{i \rightarrow \infty} c_i = 0$.  More explicitly, given such a sequence $(c_i)$, the associated character $\chi$ sends an irreducible degree $d$ polynomial $f(t)$ with root $\overline{\lambda}$ to 
\[
\prod_{(i,p) = 1} (1 + p)^{\Tr_{\ZZ_{q^d}/\ZZ_p}(c_i \lambda^i)},
\]
where $\lambda$ is the Teichm\"uller lift of $\overline{\lambda}$.
\end{Prop}

\begin{proof}
We can realize our $p$-adic character as a continuous homomorphism
\[
\prod_{(i,p) = 1} \ZZ_q \rightarrow \ZZ_p^*.
\]
By Proposition \ref{rho factor}, ignoring continuity temporarily, giving such a homomorphism is equivalent to giving a sequence of elements $(c_i)_{(i,p)=1}$.  The sequence of elements in $\prod_{(i,p) = 1} \ZZ_q$ given by $(1,0,0,\ldots), (0,1,0,\ldots),\ldots$ converges to $0$ in the product topology.  Hence the images of the maps $\ZZ_q \rightarrow \ZZ_p$ given by $\alpha \mapsto (1+p)^{\Tr(c_i \alpha)}$ should be converging $p$adically to $1$ as $i$ increases.  We claim that such an image is contained in $1+p^{j+1}\ZZ_p$ if and only if $c_i \in p^j \ZZ_q$.  The ``if'' direction is obvious.  We now prove the ``only if'' direction.  It suffices to prove the claim for $j = 0$.

Write $c_i = c_{i1}e_1 + \cdots + c_{ia}e_a$, where $e_1,\ldots,e_a$ is a basis for $\ZZ_q$ over $\ZZ_p$, and where $c_{ij} \in \ZZ_p$.  We claim that if one of the $c_{ij}$ is nonzero mod $p$, then $\Tr(c_i e_j)$ is nonzero mod $p$; this is enough to prove the claim.  Writing $\Tr(c_i e_j) = c_{i1}\Tr(e_1e_j) + \cdots + c_{ia}\Tr(e_ae_j)$, the claim follows as above from the fact that the matrix $\Tr(e_i e_j)$ is nonsingular mod $p$.

It remains only to show that $\chi$ has the explicit description in terms of the sequence $c_i$ given in the statement of the proposition.
Let $f(t) \in 1 + t\FF_q[t]$ denote an irreducible polynomial of degree $d$ with reciprocal root $\overline{\lambda}$, and let $\lambda$ denote the Teichm\"uller lift of $\overline{\lambda}$.   Then $\lambda \in \ZZ_{q^d}$.  For any $i$, the conjugates of $\lambda^i$ over $\ZZ_q$ are $\lambda^i, \lambda^{ip^a}, \lambda^{i p^{2a}}, \ldots, \lambda^{ip^{(d-1)a}}$.  As shown in Example \ref{Eg commutative diagram}, the polynomial $f_{\lambda}(t)$ corresponds to $(\Tr_{\ZZ_{q^d}/\ZZ_q}(\lambda^i))_{(i,p) = 1} \in \prod_{(i,p) = 1} \ZZ_q$.  Then by Proposition \ref{rho factor} there are \emph{unique} elements $c_i \in \ZZ_q$ such that 
\begin{align*}
\chi\left(f_{\lambda}(t)\right) &= \prod_{(i,p) = 1} (1+p)^{\Tr_{\ZZ_q/\ZZ_p}\left(c_i \Tr_{\ZZ_{q^d}/\ZZ_q}(\lambda^i)\right)} \\
&= \prod_{(i,p) = 1} (1 + p)^{\Tr_{\ZZ_q/\ZZ_p} \circ \Tr_{\ZZ_{q^d}/\ZZ_q}(c_i \lambda^i)} \\
&= \prod_{(i,p) = 1} (1 + p)^{\Tr_{\ZZ_{q^d}/\ZZ_p}(c_i \lambda^i)}.
\end{align*}
\end{proof}

\begin{Prop} \label{bijection Prop}
The correspondence described in Proposition \ref{41712Prop}, between characters and sequences, is a bijection.  
\end{Prop}

\begin{proof}
We must show that there is a bijection between continuous homomorphisms $(1 + t\FF_q[[t]])^* \rightarrow \ZZ_p^*$ and sequences $(c_i)_{(i,p) = 1}$ of elements in $\ZZ_q$ converging to $0$.  It suffices to show there is a bijection between continuous homomorphisms 
\[
\prod_{(i,p) = 1} \ZZ_q \rightarrow \ZZ_p^*
\]
and convergent sequences $(c_i)_{(i,p)=1}$.  Therefore it suffices to show there is a bijection between continuous homomorphisms $\ZZ_q \rightarrow \ZZ_p^*$ and elements $c \in \ZZ_q$.  This was shown above.  
\end{proof}

\section{Meromorphic Continuation}

For the rest of the paper, we assume $p > 2$.

In this section we will use our earlier descriptions of characters and their associated $L$-functions to address the question of when the $L$-functions are $p$-adic meromorphic.  Our main strategy is to consider our $L$-functions as being associated to certain convergent power series, and then to use results from \cite{Wan96} to study meromorphic continuation.

\subsection{The result for a special character $\chi$}
To introduce our techniques, we first consider the simplest nontrivial example in detail.  Let $q = p$.  We are going to fix a $p$-adic character $\chi: (1 + t\FF_p[[t]])^* \rightarrow \ZZ_p^*$, which by Lemma \ref{41712Prop} is the same as fixing a sequence of elements $(c_i)_{(i,p) = 1}$, where each $c_i \in \ZZ_p$, and where $\lim_{i \rightarrow \infty} c_i = 0$.  For our simple introductory case, we further assume that $c_i = 0$ for $i > 1$.  Write $1 + \pi_1 := (1 + p)^{c_1}$.

Fix an irreducible degree $d$ polynomial $f(t) \in 1 + t \FF_p[t] \subseteq (1 + t \FF_p[[t]])^*$.  Let $\overline{\lambda}$ denote a reciprocal root of $f(t)$ and let $\lambda$ denote the Teichm\"uller lift of $\overline{\lambda}$ to $\ZZ_{p^d}$.  Then for the specific character $\chi$ chosen above, by Proposition \ref{41712Prop} we have 
\[
\chi: f(t) \mapsto (1 + \pi_1)^{\Tr(\lambda)} = (1 + \pi_1)^{\lambda} (1+\pi_1)^{\lambda^p} \cdots (1 + \pi_1)^{\lambda^{p^{d - 1}}}.
\]
To this character $\chi$, we can associate an $L$-function as in Definition \ref{L-function Def} and we wish to consider the meromorphic continuation of that $L$-function.  For our techniques, it is most convenient to consider this $L$-function as also being associated to a certain power series in the variable $\lambda$, which we now describe.  

For each of the above factors we can associate its binomial power series expansion  
\[
B_{\pi_1}(\lambda^{p^i}) := (1 + \pi_1)^{\lambda^{p^i}} = \sum_{j = 0}^{\infty} \binom{\lambda^{p^i}}{j} \pi_1^j.
\]
We will consider this as a power series for which $\lambda$ is the variable and $\pi_1$ is some fixed constant in $p\ZZ_p$.  
We then have
\begin{equation} \label{61012equation}
B_{\pi_1}(\lambda) = \sum a_k \lambda^k,~a_k \in \frac{\pi_1^k}{k!} \ZZ_p.
\end{equation}
Because $v_p(k!) \leq \frac{k}{p-1}$ (see for example \cite[page 79]{Kob84}) and because $v_p(\pi_1^k) \geq k$,  we have that 
\begin{equation} \label{6612equation}
v_p(a_k) \geq k \left(\frac{p-2}{p-1}\right).
\end{equation}  The coefficients are clearly in $\QQ_p$.  Because $p > 2$, the previous inequality guarantees $a_k \in p\ZZ_p$ for $k > 0$.  In terminology to be introduced now, a power series $\sum a_k \lambda^k \in \ZZ_p[[\lambda]]$ with coefficients satisfying a growth condition as in (\ref{6612equation}) is called \emph{overconvergent}.  

\begin{Def}
Let $g(\lambda)=\sum_{k=0}^{\infty} a_k \lambda^k \in \widehat{\ZZ^{nr}_p}[[\lambda]]$.  The power series $g(\lambda)$ is called \emph{convergent} if 
$$\liminf v_p(a_k) =\infty.$$ 
The power series $g(\lambda)$ is called
\emph{overconvergent} if 
$$\liminf \frac{v_p(a_k)}{k} > 0.$$
For a positive constant $0< c \leq \infty$, the power series $g(\lambda)$ is called \emph{$c\log$-convergent} if 
$$\liminf \frac{v_p(a_k)}{\log_pk} \geq c.$$
 \end{Def}

\begin{Rmk} \label{6712Rmk}
The terms convergent and overconvergent are used because a convergent power series converges on the closed unit disk, while an overconvergent power series converges on some strictly larger disk.  A crucial example for us is that the above power series $B_{\pi_1}(\lambda)$ is overconvergent. 
\end{Rmk}

Given a convergent power series $g(\lambda)$, we now describe how to associate a character and an L-function to it.  We momentarily return to the general case $q = p^a$, because we will need the general definition in the next subsection.

\begin{Def} \label{61112Def} For $q=p^a$, 
the character $\chi$ associated to a convergent power series $g(\lambda) \in \ZZ_q[[\lambda]]$ is defined to be the unique character which sends an irreducible degree $d$ polynomial over $\mathbb{F}_q$ with reciprocal root $\overline{\lambda}$ to 
$$g(\lambda)g^{\sigma}(\lambda^p)\cdots g^{\sigma^{ad-1}}(\lambda^{p^{ad-1}}) =h(\lambda)h(\lambda^q)\cdots h(\lambda^{q^{d-1}}),$$
where $\lambda$ is the Teichm\"uller lift of $\overline{\lambda}$ and $h(\lambda)$ is the power series 
$$h(\lambda) = g(\lambda) g^{\sigma}(\lambda^p) \cdots g^{\sigma^{a-1}}(\lambda^{p^{a-1}}) \in \ZZ_q[[\lambda]],$$ and where we write $\sigma$ for the Frobenius automorphism in $\text{Gal}(\QQ_{q^d}/\QQ_p)$.
\end{Def}

\begin{Def} \label{6712Def}
The L-function of a convergent  power series $h(\lambda) \in \ZZ_q[[\lambda]]$ over $\FF_q$ is defined to be 
$$L(h/\FF_q, s) := \prod_{\overline{\lambda} \in  \overline{\FF_q}^*} \frac{1}{(1-h(\lambda)h(\lambda^q) \cdots h(\lambda^{q^{d-1}}) s^d)^{1/d}} \in 1+s \widehat{\ZZ^{nr}_p}[[s]],$$
where $\lambda$ is the Teichm\"uller lifting of $\overline{\lambda}$ and $d$ denotes the degree of $\overline{\lambda}$ over $\FF_q$. The 
associated characteristic series is defined to be 
$$C(h, s) =\prod_{k=0}^{\infty} L(h/\FF_q, q^ks).$$
\end{Def}

The following results are known about the meromorphic continuation of the L-function $L(h/\FF_q, s)$ and its characteristic series; see \cite{Wan96}. 

\begin{Thm} \label{61312Thm}
Let $h(\lambda)=\sum_{k=0}^{\infty} a_k \lambda^k \in \ZZ_q[[\lambda]]$. 
%If $g(\lambda)$ is convergent, then 
%the L-function $L(g(\lambda), s)$ is $p$-adic analytic on the closed unit disc $|s|_p \leq 1$. 
If the power series $h(\lambda)$ is \emph{overconvergent}, then 
the characteristic series $C(h/\FF_q, s)$ is entire in $|s|_p <\infty$ and thus 
the L-function $L(h/\FF_q, s)$ is 
$p$-adic meromorphic in $|s|_p < \infty$. 

More generally,  if the power series $h(\lambda)$ is  \emph{$c\log$-convergent} for some constant $0<c \leq \infty$, then 
the characteristic series $C(h/\FF_q, s)$ is entire in $|s|_p < q^c$ and thus the L-function $L(h/\FF_q, s)$ is 
$p$-adic meromorphic in the open disc $|s|_p < q^c$. 
\end{Thm}

Before considering the meromorphic continuation of the $L$-function associated to a general $p$-adic character, we return to the simple character $\chi$ fixed at the beginning of this subsection.  Recall that we are temporarily assuming $q=p$.  

The $L$-functions we are considering are related as follows:
\begin{align*}
L(\chi,s) &=  \prod_{{f(t) \text{ irred poly,} \atop f(t) \equiv 1 \bmod t}} \frac{1}{1 - \chi(f(t)) s^{\text{deg }f}} \\
&= \prod_{{f(t) \text{ irred poly,} \atop f(t) \equiv 1 \bmod t}} \frac{1}{1 - (1 + \pi_1)^{\lambda + \lambda^p + \cdots + \lambda^{p^{d-1}}} s^{d}}
\intertext{(where we write $\lambda$ for the Teichm\"uller lift of a reciprocal root of $f$, and write $d$ for the degree of $f$ over $\FF_p$)}
&= \prod_{\overline{\lambda} \in  \overline{\FF_p}^*} \frac{1}{\left(1-(1 + \pi_1)^{\lambda + \lambda^p + \cdots + \lambda^{p^{d-1}}}) s^d\right)^{1/d}}\\
&= \prod_{\overline{\lambda} \in  \overline{\FF_p}^*} \frac{1}{\left(1-B_{\pi_1}(\lambda)B_{\pi_1}(\lambda^p) \cdots B_{\pi_1}(\lambda^{p^{d-1}}) s^d\right)^{1/d}} \\
&= L(B_{\pi_1}/\FF_p,s).
\end{align*} 
Combining this equality with the preceding results, to prove that $L(\chi,s)$ is $p$-adic meromorphic for our easy introductory example, we need only demonstrate overconvergence or $\infty \log$-convergence of the power series $B_{\pi_1}(\lambda)$.  The overconvergence of this power series was already mentioned in Remark \ref{6712Rmk}.  This completes our treatment of the simple character $\chi$ we fixed at the beginning of this subsection.

\subsection{The general case}
Continue to assume $p$ is odd, but we now allow $q = p^a$ with any $a \geq 1$.
For two reasons, working with general characters $\chi: (1 + t\FF_q[[t]])^* \rightarrow \ZZ_p^*$ is more difficult than the situation in the previous subsection.  First of all, in the associated sequence $(c_i)_{(i,p) = 1}$, there will typically be infinitely many nonzero terms.  Secondly, if $q = p^a$ with $a > 1$, then the elements $c_i$ are in $\ZZ_q$, not $\ZZ_p$.  Thus the absolute traces of the elements $c_i \lambda^i$ are more complicated.  

To follow the same approach as above, we want to find a power series in $\lambda$ from which can recover the character values found in Proposition \ref{41712Prop}:
\[
\chi\left(f_{\lambda}(t)\right) = \prod_{(i,p) = 1} (1 + p)^{\Tr_{\ZZ_{q^d}/\ZZ_p}(c_i \lambda^i)}.
\]
(Here and throughout we let $f_{\lambda}(t)$ denote a general irreducible polynomial in $1 + t\FF_q[t]$.  We call its degree $d$ and we write $\lambda$ for the Teichm\"uller lift of one of its reciprocal roots. Hence $\lambda \in \ZZ_{q^d}$.)
If we again let $\sigma$ denote Frobenius in $\text{Gal}(\QQ_{q^d}/\QQ_p)$, then we can rewrite
\begin{align*}
\chi\left(f_{\lambda}(t)\right) &= \prod_{j = 0}^{ad-1}\prod_{(i,p) = 1} (1 + p)^{\sigma^{j}(c_i \lambda^i)}\\
 &= \prod_{j = 0}^{ad-1}\prod_{(i,p) = 1} (1 + p)^{\sigma^{j}(c_i) \lambda^{ip^j}}.
\end{align*}
Note now that because $c_i \in \ZZ_q$, we have  $\sigma^{j}(c_i) = c_i$ whenever $j$ is a multiple of $a$.  Let $\pi_{ij} \in \ZZ_q$ denote the unique element for which $1 + \pi_{ij} = (1 + p)^{\sigma^j(c_i)}$.  We then have
\[
\chi\left(f_{\lambda}(t)\right) = \prod_{j = 0}^{a-1} \prod_{(i,p) = 1} (1 + \pi_{ij})^{\lambda^{ip^j}}(1 + \pi_{ij})^{\lambda^{ip^jq}} \cdots (1 + \pi_{ij})^{\lambda^{ip^jq^{d-1}}}.
\]
Finally, using the same notation as in the last subsection, we have
\[
\chi\left(f_{\lambda}(t)\right) = \prod_{j = 0}^{a-1} \prod_{(i,p) = 1} B_{\pi_{ij}}(\lambda^{ip^j})B_{\pi_{ij}}(\lambda^{ip^jq}) \cdots B_{\pi_{ij}}(\lambda^{ip^jq^{d-1}}).
\]

Abbreviating the sequence of elements $\pi_{ij}$ by $\pi$, we define 
\begin{equation} \label{61712 equation}
\OO_{\pi}(\lambda) := \prod_{j = 0}^{a-1} \prod_{(i,p) = 1} B_{\pi_{ij}}(\lambda^{ip^j}).
\end{equation}
Note that this series is independent of $d$; in other words, it does not depend on the degree of $\lambda$ over $\ZZ_q$.
We now have immediately that 
$$L(\chi, s) = L( \OO_{\pi}(\lambda)/\FF_q, s), \ C(\chi, s) = C( \OO_{\pi}(\lambda)/\FF_q, s).$$
This enables us to study $L(\chi, s)$ via the series $\OO_{\pi}(\lambda).$

\begin{Thm} \label{41712Thm}
Fix a prime $p > 2$ and a prime power $q = p^a$.
Let 
\[
\chi: (1+t\FF_q[[t]])^* \rightarrow \ZZ_p^*
\]
denote a continuous character, and let $(c_i)_{(i,p) = 1}$ denote the sequence of elements in $\ZZ_q$ defined in Proposition \ref{41712Prop}.  If the series $\sum c_i x^i$ is overconvergent, then 
the characteristic series $C(\chi, s)$ is entire in $|s|_p <\infty$ and 
the associated $L$-function $L(\chi, s)$ is $p$-adic meromorphic in $|s|_p < \infty$.
\end{Thm}

\begin{proof} By Theorem \ref{61312Thm}, it suffices to prove that  $\OO_{\pi}(\lambda)$ is overconvergent. 
We assume there exists $C>0$ such that $v_p(c_i) \geq Ci$.  This is certainly true for suitably large $i$, and after shrinking $C$ we can assume it for all $i$.  

Define $R_C := \{ \sum a_k \lambda^k \mid v_p(a_k) \geq Ck \}.$
We want to show that the series $\OO_{\pi}(\lambda)$ is overconvergent.  
Since $R_C$ is a ring, 
it suffices to show that each factor $B_{\pi_{ij}}(\lambda^{ip^j}) \in R_{\frac{C}{p^{a-1}(p-1)}}$.  Write 
\[
B_{\pi_{ij}}(\lambda^{ip^j}) = \sum_{k = 0}^{\infty} a_{ijk} \lambda^{kip^j}.
\]
We know as in Equation (\ref{6612equation}) that 
\begin{align*}
v_p(a_{ijk}) &\geq v_p(\pi_{ij}^{k}) - v_p(k!) \\
&= k v_p(\pi_{ij}) - v_p(k!).
\end{align*} 
Recalling the definition of $\pi_{ij}$, and using the fact that valuation is not changed by automorphisms, $v_p(c_i) = v_p(\sigma^j(c_i))$, we have 
\begin{align*}
v_p(a_{ijk}) &\geq k (Ci+1) - \frac{k}{p-1} \\
&\geq \frac{k(Ci + (p-2))}{p-1} \\
&\geq \frac{kip^jCi}{ip^j(p-1)} \\
&\geq \frac{C}{p^{a-1}(p-1)} kip^j,
\end{align*}
as required.
\end{proof}

\begin{Thm} \label{log Thm}
Keep notation as in Theorem \ref{41712Thm}.  If the series $\sum c_i x^i$ is $C\log$-convergent, then 
the characteristic series $C(\chi, s)$ is entire in the disk $|s|_p <q^C$ and the $L$-function $L(\chi, s)$ is $p$-adic meromorphic in the disc $|s|_p < q^C$.
\end{Thm}

\begin{proof}
It again suffices to show that $\OO_{\pi}(\lambda)$ is $C\log$-convergent.
Fix any $\epsilon$ such that $\frac{C}{2} > \epsilon > 0$.  By our assumption on $\sum c_i x^i$, there exists a constant $M(\epsilon)$ such that for all $i > M(\epsilon)$, we have $v_p(c_i) \geq (C-\epsilon)(\log_p(i) + 1)$.  Consider $S_C \subseteq 1 + \lambda \ZZ_q[[\lambda]]$ defined by
\[ 
S_C = \left\{ 1+ \sum_{k=1}^{\infty} a_k \lambda^k \mid a_k \in \ZZ_q,  \ v_p(a_k)\geq C (\log_p(k) + 1) \right\}.
\]  An easy computation shows that $S_C$ is a ring.  Our strategy is to show that if $i$ is suitably large, then $B_{\pi_{ij}}(\lambda^{ip^j}) \in S_{C - 2\epsilon}$ and that for all other $i$, then $B_{\pi_{ij}}(\lambda^{ip^j})$ is overconvergent.

We proceed as in the proof of Theorem \ref{41712Thm}. Write 
\[
B_{\pi_{ij}}(\lambda^{ip^j}) = 1 + \sum_{k = 1}^{\infty} a_{ijk} \lambda^{kip^j}.
\]
We have again
\begin{align*}
v_p(a_{ijk}) &\geq v_p(\pi_{ij}^{k}) - v_p(k!) \\
&= k v_p(\pi_{ij}) - v_p(k!).
\end{align*}
Now we have our first departure from the proof of Theorem \ref{41712Thm}, because we now have a weaker growth condition on the sequence $(c_i)$, and hence a weaker growth condition on the terms $\pi_{ij}$, which have valuation $v_p(c_i) + 1$.  In our current case $i > M(\epsilon)$, we have
\begin{align*}
v_p(a_{ijk}) &\geq k (C-\epsilon)(\log_p(i) + 1) + k - \frac{k}{p-1} \\ 
&\geq  (C - \epsilon)(\log_p(i) + \log_p(k))\\
&= (C - \epsilon) (\log_p(i) + \log_p(k) + \log_p(p^j) + 1) - (C-\epsilon)(j+1)\\
& = (C - \epsilon) \log_p(kip^j) - (C-\epsilon) (j+1)\\
&\geq \left(C - 2\epsilon\right) \log_p(kip^j) + \epsilon \log_p(i) - C(j+1).
\end{align*}
Because $C$ and $\epsilon$ are fixed, and because $j$ is bounded, we have that for all but finitely many $i$, 
\[
v_p(a_{ijk}) \geq \left(C - 2\epsilon\right)(\log_p(kip^j)+1).
\]
Hence for almost all values of $i$, we have $B_{\pi_{ij}}(\lambda^{ip^j}) \in S_{C - 2\epsilon}$.

Now consider the finitely many remaining values of $i$.  We can find $D>0$ such that $v_p(\pi_{ij}) \geq D(i+1)$ for all remaining $i$ and all $j$.  We can then use the proof of Theorem \ref{41712Thm} to see that all the corresponding series $B_{\pi_{ij}}$ are overconvergent.  

It is clear that the product of a $C\log$-convergent power series and an overconvergent series is again $C\log$-convergent.  We thus have that $\OO_{\pi}(\lambda)$ is $(C-2\epsilon)\log$-convergent for every $\epsilon$ in the range $\frac{C}{2} > \epsilon > 0$.  It follows that $\OO_{\pi}(\lambda)$ is $C\log$-convergent.
\end{proof}

\section{Converting from a power series to a sequence}

We have a bijection between characters $\chi$ and sequences $(c_i)$; see Proposition \ref{bijection Prop}.  On the other hand, many power series can induce the same character; for example, any power series of the form $g(\lambda)/g^{\sigma}(\lambda^p)$ induces the trivial character.  Thus the function 
\[
(c_i) \mapsto \prod_{(i,p) = 1} (1 + p)^{c_i \lambda^i}
\] cannot be a bijection between sequences and power series.  Write $F$ to denote this function.  Write $g(\lambda)$ for the image of $(c_i)$ under $F$.  Then, in the notation of Equation (\ref{61712 equation}), we have $\OO_{\pi}(\lambda) = g(\lambda)g^{\sigma}(\lambda^p) \cdots g^{\sigma^{a-1}}(\lambda)^{p^{a-1}}$.  The point of the function $F$ is that the character associated to $(c_i)$ as in Proposition \ref{41712Prop} is the same as the character associated to $g(\lambda)$ as in Definition \ref{61112Def}.

Our goal in this section is to describe a one-sided inverse $G$ from convergent power series in $\ZZ_q\langle \lambda \rangle$ to sequences of elements in $\ZZ_q$,  written $G: g(\lambda) \mapsto (d_i)$, with the following two properties:
\begin{enumerate}
\item The composition $G \circ F$ is the identity.
\item If $g(\lambda)$ is $C\log$-convergent, then so is $G(g(\lambda))$.
\end{enumerate}

How should we define $G$?  Given a power series $g(\lambda)$, we would like to find a sequence $(d_i)_{(i,p) = 1}$ such that 
\[
g(\lambda) = (1 + p)^{\sum d_i \lambda^i}.
\]
Let $\plog$ denote the $p$-adic logarithm; see for instance \cite[Chapter IV]{Kob84}.  (Note the difference in notation between this and the base $p$ logarithm $\log_p$ used to define $\log$-convergence in the last section.)  Applying $\plog$ to both sides of the above equation, we find 
\[
\frac{\plog(g(\lambda))}{\plog(1+p)} = \sum d_i \lambda^i.
\]
The problem is that the left side probably includes nonzero terms $d_i \lambda^i$ even when $(i,p) \neq 1$.  The following definition helps us remove unwanted $p$-powers.

\begin{Def} \label{psi Def}
Define a map $\psi_p: \ZZ_q\langle \lambda \rangle \rightarrow \ZZ_q \langle\lambda \rangle$ by sending 
\[
\psi_p: \sum_{k = 0}^{\infty} a_k \lambda^k \mapsto \sum_{(i,p) = 1} b_i \lambda^i, \text{ where } b_i = \sum_{j = 0}^{\infty} \sigma^{-j}(a_{ip^j}).
\]
\end{Def}

\begin{Rmk}
Note that our definition only makes sense for convergent power series.  Also note that $\psi_p$ is additive but not linear.
\end{Rmk}

\begin{Lem}
If $g(\lambda)$ is overconvergent (respectively, $C\log$-convergent), then $\psi_p(g(\lambda))$ is overconvergent (respectively, $C\log$-convergent).
\end{Lem}

\begin{proof}
This is obvious.  Note that for both of these conditions, the requirement on the coefficient of $\lambda^i$ is stricter when $i$ is larger.  Because $\psi_p$ potentially decreases the $i$ exponents, it will preserve overconvergence and $C\log$-convergence.
\end{proof}

The following lemma is the reason Definition \ref{psi Def} is useful.

\begin{Lem}
Let $c(x) \in  \ZZ_{q}\langle x \rangle$ denote a convergent power series.  Let $\overline{\lambda} \in \overline{\FF_p}$ denote some nonzero element of degree $d$ over $\FF_q$, let $\lambda$ denote its Teichm\"uller lift, and let $\Tr$ denote the absolute trace from $\ZZ_{q^d}$ to $\ZZ_p$.  Then we have 
\[
\Tr(c(\lambda)) = \Tr(\psi_p(c)(\lambda)).
\]
\end{Lem}

\begin{proof}
Because trace is additive (and $p$-adically continuous), it suffices to show that $\Tr(\sigma^{-1}(c_i) \lambda^{i}) = \Tr(c_i \lambda^{pi})$ for all $i$, 
which is obviously true by the definition of absolute trace. 
\end{proof}

\begin{Lem} \label{61212Lem}
Let $\sum c_i \lambda^i \in 1 + p\lambda \ZZ_q[[\lambda]]$ denote a $C\log$-convergent series.  Then $\plog \left( \sum c_i \lambda^i \right)$ is also $C\log$-convergent.
\end{Lem}

\begin{proof}
By our assumption, for every $\epsilon$ in the range $C > \epsilon > 0$, there exists a positive integer constant $M(\epsilon)  \geq 2$ such that $v_p(c_k) \geq (C- \epsilon) \log_p(k) + 1$ for all $k \geq M(\epsilon)$.  (The reason for the ``$+1$'' will become apparent later.)  Write 
\[
\plog \left( \sum c_k \lambda^k \right) =: \sum d_k \lambda^k.
\]
Our goal is to show that for every $\epsilon$ there exists a constant $N(\epsilon)$ such that $v_p(d_k) \geq (C - \epsilon)\log_p(k)$ whenever $k> N(\epsilon)$.

Using the Taylor series expansion for $\plog$, we have 
\[
\sum d_k \lambda^k = \sum_{k = 1}^{\infty} ~\sum_{k_1 + 2k_2 + \cdots + rk_r = k} \frac{(-1)^{k_1+\cdots +k_r-1}}{k_1 + k_2 + \cdots + k_r} \binom{k_1+\cdots +k_r}{k_1, k_2,\ldots,k_r} c_1^{k_1} \cdots c_r^{k_r} \lambda^k.
\]
Because multinomial coefficients are integers, we have
\[
v_p(d_k) \geq k_1 v_p(c_1) + \cdots + k_r v_p(c_r) - v_p(k_1 + \cdots + k_r).
\]
Note that we always have $v_p(c_i) \geq 1$.  Abbreviate the index $M(\epsilon)\geq 2$ chosen above by $s$.  Then we have 
\[
v_p(d_k) \geq k_1 + \cdots + k_{s-1} + \left[k_s(C-\epsilon)\log_p(s) + k_s\right] + \cdots + \left[k_r (C-\epsilon)\log_p(r)+k_r\right] \] 
 \[ - v_p(k_1 + k_2 + \cdots + k_r) \] 
 \[ \geq \frac{1}{2}(k_1+\cdots +k_r) + (C-\epsilon) \sum_{j=s}^r  k_j\log_p j\] 
 \[ \geq \frac{1}{2}(k_1+\cdots +k_r) + (C-\epsilon) \log_p (\sum_{j=s}^r  j k_j )\]
 If  $\sum_{j=s}^r  j k_j  \geq k/2$, then 
\[ v_p(d_k) \geq (C-\epsilon) \log_p (\frac{1}{2} k) = (C-\epsilon) \log_p k   -  (C-\epsilon) \log_p 2. \] 
If $\sum_{j=s}^r  j k_j  < k/2$, then $\sum_{j=1}^{s-1}  j k_j  \geq k/2$ and thus 
\[ v_p(d_k) \geq  \frac{1}{2}(k_1+\cdots +k_{s-1}) \geq \frac{1}{4(s-1)} k > (C-\epsilon) \log_p k \] 
for all sufficiently large $k$. This proves that 
\[ \lim_k \inf \frac{v_p(d_k)}{\log_p k} \geq C-\epsilon.\]
The theorem is proved. 

\end{proof}

When we assemble the above results, we attain the following result.  It describes the function $G$ promised at the beginning of this section.

\begin{Prop} \label{61212Prop}
Let $g(\lambda)$ denote a $C\log$-convergent series.  Convert this into the power series $\psi_p \circ \plog (g(\lambda))$, and let $d_i$ denote the coefficient of $\lambda^i$ in the new power series.  Then the sequence $(d_i)_{(i,p) = 1}$ is $C\log$-convergent and its associated $L$-series is the same as the $L$-series associated to $g(\lambda)$.  
\end{Prop}

We close this section with a special example.

\begin{Eg} 
Let $q = p$ and let $C > 0$ denote some constant. 
Power series were constructed in \cite[Theorem 1.2]{Wan96} which were $C\log$-convergent and whose associated $L$-functions failed to have meromorphic continuation to the disk $|s| < p^{C + \epsilon}$ for any $\epsilon$.  We briefly mention some implications in our context.   Define
\[
g_C(\lambda) = 1 + \sum_{i \geq 1} p^{Ci+1}u_i \lambda^{p^{i}-1}, 
\]
where $u_i \in \mathbb{Z}$ is such that the reduction modulo $p$ of $\sum_i u_i t^i$ is not in $\FF_p(t)$. 
This series $g_C(\lambda)$ is clearly $C\log$-convergent, so by Proposition \ref{61212Prop}, the associated sequence $(d_i)_{(i,p) = 1}$ is $C\log$-convergent.  On the other hand, we know by \cite{Wan96} that the associated $L$-function is meromorphic in the disk $|s|_p <p^C$ but not meromorphic 
in any larger disk $|s|_p < p^{C+\epsilon}$ for any $\epsilon > 0$.  Then by Theorem \ref{log Thm}, the associated sequence $(d_i)_{(i,p) = 1}$ is not $(C+\epsilon)$-convergent.
\end{Eg}

\subsection*{Acknowledgments} The authors thank James Borger, Lars Hesselholt, Kiran Kedlaya, Jack Morava, Tommy Occhipinti, and Liang Xiao for many useful discussions.

\bibliography{padicChar}
\bibliographystyle{plain}

\end{document}